\documentclass{amsart}
\usepackage{amsmath,amssymb}
\usepackage[dvips]{graphics}
\usepackage[all]{xy}
\newtheorem{Theorem}{Theorem}[section]
\newtheorem{Lemma}[Theorem]{Lemma}
\newtheorem{Proposition}[Theorem]{Proposition}
\newtheorem{Corollary}[Theorem]{Corollary}

\theoremstyle{Definition}

\newtheorem{Example}[Theorem]{Example}

\theoremstyle{Remark}

\makeatletter
\@addtoreset{figure}{section}
\def\@thmcountersep{-}
\makeatother

\numberwithin{equation}{section}

\begin{document}

\title{On the Wu invariants for immersions of a graph into the plane}

\author{Ryo Nikkuni}
\address{Department of Mathematics, School of Arts and Sciences, Tokyo Woman's Christian University, 2-6-1 Zempukuji, Suginami-ku, Tokyo 167-8585, Japan}
\email{nick@lab.twcu.ac.jp}

\subjclass{Primary 57Q35; Secondary 57M15}

\date{}

\dedicatory{}

\keywords{Immersion, Graph, Wu invariant}

\begin{abstract}
We give an explicit calculation of the Wu invariants for immersions of a finite graph into the plane and classify all generic immersions of a graph into the plane up to regular homotopy by the Wu invariant. This result is a generalization of the fact that two plane curves are regularly homotopic if and only if they have the same rotation number. 
\end{abstract}

\maketitle

\section{Introduction}

Throughout this paper we work in the piecewise linear category. In \cite{wu2}, \cite{wu1}, Wu defined an isotopy invariant of embeddings and immersions of polyhedra into the Euclidean space in terms of the cohomology of deleted product spaces. In case of embeddings, this invariant classifies all embeddings of a graph into ${\mathbb R}^{3}$ up to spatial graph-homology (Taniyama \cite{tani2}). But as far as the author knows, little is known about an application of this invariant in case of immersions. Our purpose in this paper is to give an explicit calculation of Wu's invariant of immersions of a graph into ${\mathbb R}^{2}$ and apply it to a geometric classification. 

Let $G$ be a finite, connected and simple graph which has at least one edge. We denote the set of all vertices (resp. edges) of $G$ by $V(G)$ (resp. $E(G)$). If the terminal vertices of an edge $e$ of $G$ are $u$ and $v$, then we denote $e=(u,v)=(v,u)$. We denote the number of edges incident to a vertex $v$ by ${\rm deg}v$. Note that $G$ has a structure of a finite $1$-dimensional simplicial complex. We regard $G$ as a topological space by considering its geometric realization, namely $G$ is a compact and connected $1$-dimensional polyhedron. In this situation, each of the vertices and the edges of $G$ can be regarded as a subset of $G$. We call a continuous map $f:G\to {\mathbb R}^{2}$ a {\it plane immersion} of $G$ if there exists an open covering $\{U_{\nu}\}$ of $G$ such that $f{|}_{U_{\nu}}$ is an embedding for any $\nu$. A plane immersion $f$ of $G$ is said to be {\it generic} if all of its multipoints are transversal double points away from vertices. We say that two plane immersions $f$ and $g$ of $G$ are {\it regularly homotopic} if there exists a homotopy $F:G\times [0,1]\to {\mathbb R}^{2}$ from $f$ to $g$ and an open covering $\{U_{\nu}\}$ of $G$ such that ${f}_{t}{|}_{U_{\nu}}$ is an embedding for any $\nu$ and for any $t\in [0,1]$, where $f_{t}$ is a continuous map from $G$ to ${\mathbb R}^{2}$ defined by $f_{t}(x)=F(x,t)$ for any $x\in G$. Note that regular homotopy defines an equivalence relation on plane immersions of a graph.\footnote{This equivalence relation was introduced in \cite{wu1} by the name of {\it local isotopy}.}

We give the precise definition of the Wu invariant ${\mathcal R}(f)$ of a plane immersion $f$ of a graph in the next section and also give an explicit calculation of ${\mathcal R}(f)$ in section $3$. It can be calculated as a first cohomology class of a subspace of the symmetric deleted product of the graph, which is called the {\it symmetric tube} of the graph. Moreover we have the following classification theorem. 

\begin{Theorem}\label{main} 
Let $f$ and $g$ be two generic plane immersions of a graph $G$. Then the following are equivalent: 
\begin{enumerate}
\item $f$ and $g$ are regularly homotopic. 
\item $f$ and $g$ are transformed into each other by 
the local moves as illustrated in Fig. \ref{rhomo2} (1), (2), (3) and ambient isotopies. 
\item ${\mathcal R}(f)={\mathcal R}(g)$. 
\end{enumerate}
\end{Theorem}
\begin{figure}[htbp]
      \begin{center}
\scalebox{0.45}{\includegraphics*{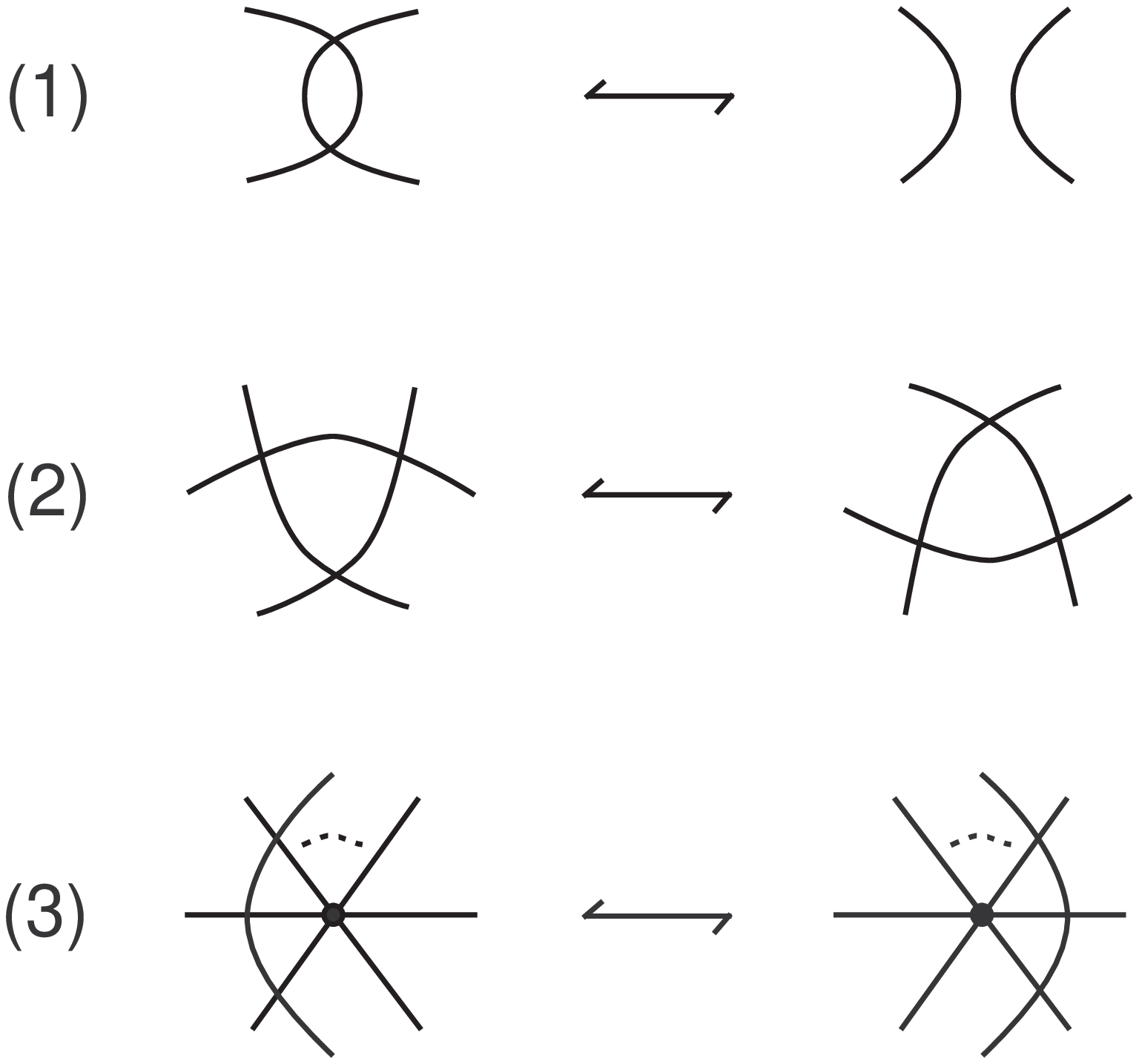}}
      \end{center}
   \caption{}
  \label{rhomo2}
\end{figure}

Let $K_{m}$ be the {\it complete graph} on $m$ vertices for a positive 
integer $m$, namely $V(K_{m})=\{v_{1},v_{2},\ldots,v_{m}\}$ and $E(K_{m})=\{(v_{i},v_{j})~(1\le i<j\le m)\}$. A plane immersion $f$ of $K_{3}$ is called a {\it plane curve}. By Theorem \ref{main}, we have the following corollary. 

\begin{Corollary}\label{maincor} 
Let $f$ and $g$ be two generic plane curves. Then the following are equivalent: 
\begin{enumerate}
\item $f$ and $g$ are regularly homotopic. 
\item $f$ and $g$ are transformed into each other by 
the local moves as illustrated in Fig. \ref{rhomo2} (1), (2) and ambient isotopies.
\item ${\mathcal R}(f)={\mathcal R}(g)$. 
\end{enumerate}
\end{Corollary}

We prove Theorem \ref{main} in section $4$. As we will see in Example \ref{circle}, ${\mathcal R}(f)$ of a plane curve $f$ coincides with the {\it rotation number} \cite{whitney} of $f$. (In this paper we consider the orientation on ${\mathbb R}^{2}$ with positive rotation numbers in the counterclockwise direction.) Thus Corollary \ref{maincor} coincides with the regular homotopy classification of plane curves by Whitney-Graustein's theorem \cite{whitney} and Kauffman's combinatorial interpretation \cite{kauffman}. We remark here that recently Permyakov gives a simple combinatorial interpretation of ${\mathcal R}(f)$ and shows a theorem which corresponds to Theorem \ref{main} \cite{perm}.

\section{Wu invariant} 

In this section we give the definition of the Wu invariant of a plane immersion of a graph $G$. We refer the reader to \cite{wu1} for the general case. Let $X$ be a topological space. For the embedding $\tilde{d}:X\to X\times X$ defined by $\tilde{d}(x)=(x,x)$, we call $\widetilde{X}^{*}=(X\times X)\setminus \tilde{d}(X)$ the {\it deleted product} of $X$. A map $\sigma(x,y)=(y,x)$ gives a free ${\mathbb Z}_{2}$-action on $\widetilde{X}^{*}$. We call ${X}^{*}={\widetilde{X}^{*}}/{\mathbb Z}_{2}$ the {\it symmetric deleted product} of $X$. We denote the image of $\tilde{d}(X)$ by the natural projection from $X\times X$ to $(X\times X)/{\mathbb Z}_{2}$ by ${d}(X)$. Let $U$ be a neighborhood of $\tilde{d}(X)$ in $X\times X$. Then $\widetilde{U}^{*}=U \setminus \tilde{d}(X)$ is called a {\it deleted neighborhood of $\tilde{d}(X)$ in $\widetilde{X}^{*}$}. A deleted neighborhood $\widetilde{U}^{*}$ is said to be {\it $\sigma$-invariant} if $\sigma(U)=U$. Then we call ${U}^{*}={\widetilde{U}^{*}}/{\mathbb Z}_{2}$ a {\it symmetric deleted neighborhood of ${d}(X)$ in ${X}^{*}$}. 
 
For a graph $G$, let $\{U_{\lambda}^{*}\}$ be the set of all symmetric deleted 
neighborhoods of ${d}(G)$ in ${G}^{*}$. Then $\{U_{\lambda}^{*},\prec\}$ forms an oriented set by $U_{\lambda}^{*}\prec U_{\mu}^{*}$ if $U_{\lambda}^{*}\supset
U_{\mu}^{*}$. For this oriented set, $\{H^{1}(U_{\lambda}^{*};{\mathbb Z}),{i_{\lambda}^{\mu}}^{*}\}$ forms an inductive system of modules, where $H^{1}(\cdot;{\mathbb Z})$ denotes the integral first cohomology group and 
\begin{eqnarray*}
{i_{\lambda}^{\mu}}^{*}:
H^{1}(U_{\lambda}^{*};{\mathbb Z})\to H^{1}(U_{\mu}^{*};{\mathbb Z})
\end{eqnarray*}
is a homomorphism induced by the inclusion. Then we denote the inductive limit $\displaystyle {\lim_{\longrightarrow}}~H^{1}({U}_{\lambda}^{*};{\mathbb Z})$ by $R(G)$. We note that we have the following natural homomorphism 
\begin{eqnarray}
i_{\lambda}^{*}:H^{1}({U}_{\lambda}^{*};{\mathbb Z})
\longrightarrow R(G) \label{natural}
\end{eqnarray}
for any symmetric deleted neighborhood ${U}_{\lambda}^{*}$ of $d(G)$ in $G^{*}$. 

Let $f:G\to {\mathbb R}^{2}$ be a plane immersion. Namely there exists an open covering ${\mathcal U}=\{U_{\nu}\}$ of $G$ such that $f{|}_{U_{\nu}}$ is an embedding for any $\nu$. Then the set 
\begin{eqnarray*}
\widetilde{W}_{\mathcal U}=\{(x_{1},x_{2})\in \widetilde{G}^{*}~|~
{\rm there~exists~a~}U_{\nu}~{\rm such~that~}x_{1},x_{2}\in U_{\nu}\}
\end{eqnarray*}
forms a $\sigma$-invariant 
deleted neighborhood of $\tilde{d}(G)$ in $\widetilde{G}^{*}$ and a continuous map $\bar{f}:W_{\mathcal U}\to ({\mathbb R}^{2})^{*}$ is defined by $\bar{f}[x_{1},x_{2}]=[f(x_{1}),f(x_{2})]$. On the other hand, it is well known that a continuous map $r:({\widetilde{\mathbb R}}^{2})^{*}\to {\mathbb S}^{1}$ defined by $r(y_{1},y_{2})=(y_{1}-y_{2})/||y_{1}-y_{2}||$ is a $\sigma$-equivariant strong deformation retract and ${r}:({\mathbb R}^{2})^{*}\to {\mathbb S}^{1}/{\mathbb Z}_{2}\approx {\mathbb S}^{1}$ is also a strong deformation retract, where ${\mathbb S}^{1}/{\mathbb Z}_{2}$ denotes the quotient space of ${\mathbb S}^{1}$ by identifying the antipodal points. Let $\Sigma$ be a generator of $H^{1}({\mathbb S}^{1};{\mathbb Z})\cong {\mathbb Z}$. Then the image of $\Sigma$ by the composition 
\begin{eqnarray*}
H^{1}({\mathbb S}^{1};{\mathbb Z})\stackrel{\stackrel{{r}^{*}}{\cong}}
{\longrightarrow} 
H^{1}({{\mathbb R}^{2}}^{*};{\mathbb Z})\stackrel{{\bar{f}}^{*}}{\longrightarrow}
H^{1}({W}_{\mathcal U};{\mathbb Z})\stackrel{i_{{\mathcal U}}^{*}}{\longrightarrow} 
R(G)
\end{eqnarray*}
is denoted by ${\mathcal R}(f)$, where $i_{{\mathcal U}}^{*}$ is the natural homomorphim of (\ref{natural}) for ${W}_{\mathcal U}$. We call ${\mathcal R}(f)$ a {\it Wu invariant} \footnote{This invariant was introduced in \cite{wu1} by the name of {\it local isotopy class} and denoted by $\Lambda_{f}(G)$.} of $f$. We remark here that the definition above is independent of the choice of ${\mathcal U}$.  
\begin{Proposition}{\rm (\cite{wu1})}\label{inva}  
${\mathcal R}(f)$ is a regular homotopy invariant. 
\end{Proposition}

\begin{proof}
Let $f$ and $g$ be two regularly homotopic plane immersions of $G$. Namely there exists a homotopy $F:G\times [0,1]\to {\mathbb R}^{2}$ from $f$ to $g$ and an open covering $\{U_{\nu}\}$ of $G$ such that ${f}_{t}{|}_{U_{\nu}}$ is an embedding for any $\nu$ and for any $t\in [0,1]$, where $f_{t}(x)=F(x,t)$ for $x\in G$. Then we can define a homotopy $F_{\mathcal U}:W_{\mathcal U}\times [0,1]\longrightarrow ({\mathbb R}^{2})^{*}$ from $\bar{f}$ to $\bar{g}$ by $F_{\mathcal U}([x_{1},x_{2}],t)=[f_{t}(x_{1}),f_{t}(x_{2})]$. Thus we have that 
${\mathcal R}(f)=
i_{{\mathcal U}}^{*}{\bar{f}}^{*}{\bar{r}}^{*}(\Sigma)=
i_{{\mathcal U}}^{*}{\bar{g}}^{*}{\bar{r}}^{*}(\Sigma)=
{\mathcal R}(g)$. 
This completes the proof. 
\end{proof}

\section{Symmetric tube of a graph} 

A precise method to calculate $R(G)$ is provided in \cite{wu1}. Let $X$ and $Y$ be two topological spaces and $M=X\cup (X\times Y\times [0,1])\cup Y$ the disjoint union. Let us consider a quotient space by identifying $(x,y,0)\in X\times Y\times [0,1]$ with $x\in X$ and $(x,y,1)\in X\times Y\times [0,1]$ with $y\in Y$. We call the quotient space a {\it join of $X$ and $Y$} and denote it by $X{\circ}Y$. We set $[X,Y]^{(0)}=\{[x,y,1/2]\in X{\circ}Y~|~x\in X,~y\in Y\}$. For example, the join $v{\circ}e$ of a vertex $v$ and an edge $e$ is homeomorphic to a $2$-simplex, and $[v,e]^{(0)}$ is homeomorphic to a $1$-simplex. The following is a special case of what is called a {\it canonical cellular decomposition} of the product space of $X$ \cite{wu1}, \cite{hu}. 

\begin{Proposition}\label{decomp}  
Let $G$ be a graph. Then $G\times G$ is decomposed into the following cells: 
\begin{enumerate}
\item $\tilde{d}(s)$ for $s\in V(G)$ or $E(G)$. 
\item $s_{1}\times s_{2}$ for $s_{i}\in V(G)$ or $E(G)\ (i=1,2)$ and $s_{1}\cap s_{2}=\emptyset$.  
\item $\tilde{d}(s)\circ (s_{1}\times s_{2})$ for $s,s_{1},s_{2}\in V(G)$ or $E(G)$, $s_{1}\cap s_{2}=\emptyset$ and $s\cup s_{i}$ is contained in a vertex or an edge of $G\ (i=1,2)$. 
\end{enumerate}
\end{Proposition}

In particular, the cellular complex which consists of all cells $[s,s_{1}\times s_{2}]^{(0)}$ for simplices $s$, $s_{1}$ and $s_{2}$ of Proposition \ref{decomp} (3) is called a {\it tube} of $G$ and is denoted by $\widetilde{G}^{(0)}$. Clearly $\widetilde{G}^{(0)}$ is $\sigma$-invariant in $\widetilde{G}^{*}$. We call $\widetilde{G}^{(0)}/{\mathbb Z}_{2}$ a {\it symmetric tube} of $G$ and denote it by $G^{(0)}$. We denote $[s,s_{1}\times s_{2}]^{(0)}/{\mathbb Z}_{2}$ by $[s,s_{1}*s_{2}]$, and we note that $[s,s_{1}*s_{2}]=[s,s_{2}*s_{1}]$. 

\begin{Example}\label{cell_dec_ex}
{\rm Let $K_{3}$ be the complete graph on three vertices as illustrated in Fig. \ref{cell_decomp}. The figure on the right side in Fig. \ref{cell_decomp} illustrates the canonical cellular decomposition of $K_{3}\times K_{3}$ in the sense of Proposition \ref{decomp} as an expanded diagram of the $2$-dimensional torus. The dotted thick parts and black thick parts represent the cells of Proposition \ref{decomp} (1) and (2), respectively. The gray thick parts represent ${\widetilde{K}}_{3}^{(0)}$. Thus we can see that $K_{3}^{(0)}$ is homeomorphic to the circle. 
}
\end{Example}
\begin{figure}[htbp]
      \begin{center}
\scalebox{0.425}{\includegraphics*{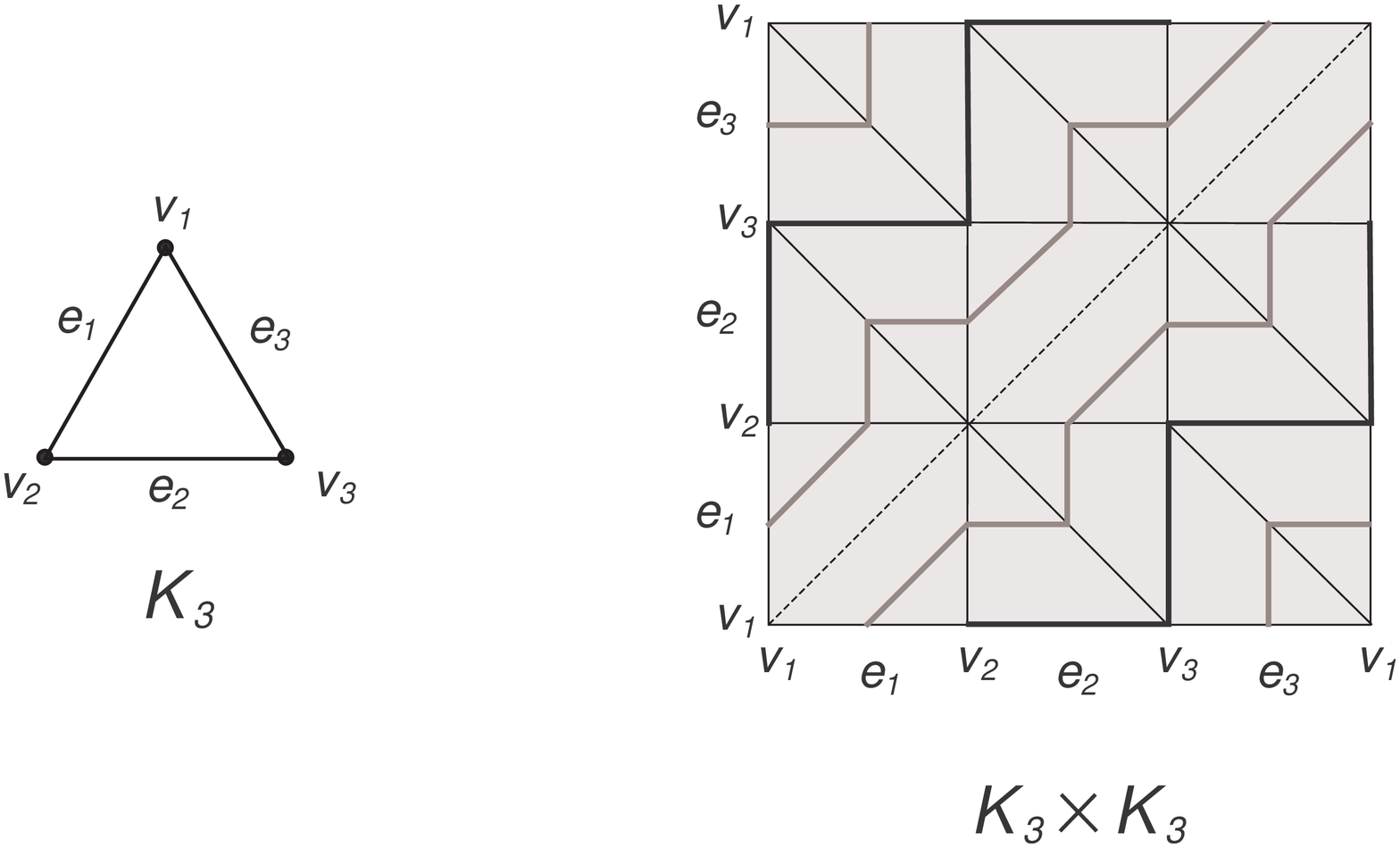}}
      \end{center}
   \caption{}
  \label{cell_decomp}
\end{figure}

Let $P(\widetilde{G}^{*})$ be the cellular complex which consists of all cells $s_{1}\times s_{2}$ for simplices $s_{1},s_{2}$ of Proposition \ref{decomp} (2). Clearly $P(\widetilde{G}^{*})$ is also $\sigma$-invariant in $\widetilde{G}^{*}$. We denote $P(\widetilde{G}^{*})/{\mathbb Z}_{2}$ by $P({G}^{*})$. We note that $G^{*}\setminus P({G}^{*})$ is a symmetric deleted neighborhood of ${d}(G)$ in ${G}^{*}$. It is known that  
\begin{eqnarray*}
i_{G^{*}\setminus P({G}^{*})}^{*}:H^{1}(G^{*}\setminus 
P({G}^{*});{\mathbb Z})\stackrel{\cong}{\longrightarrow}  R(G), 
\end{eqnarray*}
where $i_{G^{*}\setminus P({G}^{*})}^{*}$ is the natural homomorphism of (\ref{natural}) for $G^{*}\setminus P({G}^{*})$, and there exists a deformation retract $j:G^{*}\setminus P({G}^{*})\to G^{(0)}$ \cite{wu1}. Therefore we have the following. 

\begin{Theorem}{\rm (\cite{wu1})}\label{iso}  
$i_{G^{(0)}}^{*}=i_{G^{*}\setminus P({G}^{*})}^{*} j^{*}:H^{1}(G^{(0)};{\mathbb Z})\stackrel{\cong}{\longrightarrow} R(G)$. 
\end{Theorem}

Thus the calculation of $H^{1}(G^{(0)};{\mathbb Z})$ provides a precise method to calculate $R(G)$. To calculate $H^{1}(G^{(0)};{\mathbb Z})$, we investigate the structure of $G^{(0)}$ directly. We set $V(G)=\{v_{1},v_{2},\ldots,v_{m}\}$ and $E(G)=\{e_{1},e_{2},\ldots,e_{n}\}$. We choose a fixed orientation on each edge of $G$. We put 

\noindent
$Z_{st}^{s}=Z_{ts}^{s}=[v_{s},v_{s}*v_{t}]$ for $(v_{s},v_{t})\in E(G)$, 

\noindent
$W_{st}^{u}=W_{ts}^{u}=[v_{u},v_{s}*v_{t}]$ for $(v_{u},v_{s}),(v_{u},v_{t})\in E(G)$, $v_{s}\neq v_{t}$, 

\noindent
$X_{st}^{i}=X_{ts}^{i}=[e_{i},v_{s}*v_{t}]$ for $e_{i}=(v_{s},v_{t})\in E(G)$, and 

\noindent
$Y_{ti}^{s}=[v_{s},v_{t}*e_{i}]$ for $(v_{s},v_{t}),e_{i}\in E(G)$, $(v_{s},v_{t})\neq e_{i}$ and $v_{s}\subset e_{i}$. 

\noindent
Note that $Z_{st}^{s}$ and $W_{st}^{u}$ are $0$-dimensional simplices of $G^{(0)}$, and $X_{st}^{i}$ and $Y_{ti}^{s}$ are $1$-dimensional simplices of $G^{(0)}$. An orientation of $X_{st}^{i}$ is induced by $e_{i}$, and an orientation of $Y_{ti}^{s}$ is induced by $e_{i}$. We can consider $Z_{st}^{s}$ and $W_{st}^{u}$
as $0$-chains in $C_{0}(G^{(0)};{\mathbb Z})$ and $X_{st}^{i}$ and $Y_{ti}^{s}$ as $1$-chains in $C_{1}(G^{(0)};{\mathbb Z})$. The dual cochain of $Z_{st}^{s}$, $W_{st}^{u}$, $X_{st}^{i}$ and $Y_{ti}^{s}$ are denoted by 
$Z_{s}^{st}$, $W_{u}^{st}$, $X_{i}^{st}$ and $Y_{s}^{ti}$, respectively. It is not difficult to see the following. 

\begin{Proposition}\label{decomp2}  
For a graph $G$, a cell of symmetric tube $G^{(0)}$ is one of $Z_{st}^{s}$, $W_{st}^{u}$, $X_{st}^{i}$ and $Y_{ti}^{s}$ as above. Therefore $G^{(0)}$ is also a graph. 
\end{Proposition}

For example, let $S_{n}$ be a graph as illustrated in Fig. \ref{startuben}. By enumerating vertices and edges of $S_{n}^{(0)}$ and observing the connection between them directly, we have the following.

\begin{Lemma}\label{stube}  
The symmetric tube $S_{n}^{(0)}$ of $S_{n}$ is a graph as follows: 
\begin{enumerate}
\item $V(S_{n}^{(0)})
=\{Z_{n+1,i}^{i}\ {\rm and}\ Z_{n+1,i}^{n+1}\ (i=1,2,\ldots,n),\ W_{jk}^{n+1}\ (1\le j<k\le n)\}$. 
\item $E(S_{n}^{(0)})=\{
X_{n+1,i}^{i}\ (i=1,2,\ldots,n),\ 
Y_{jk}^{n+1}\ {\rm and}\ Y_{kj}^{n+1}\ (1\le j<k\le n)\}$. 
\item 
$X_{n+1,i}^{i}=(Z_{n+1,i}^{n+1},Z_{n+1,i}^{i})\ (i=1,2,\ldots,n)$, 

\noindent
$Y_{jk}^{n+1}=(Z_{n+1,j}^{n+1},W_{jk}^{n+1})$ {\rm and} $Y_{kj}^{n+1}=(Z_{n+1,k}^{n+1},W_{jk}^{n+1})\ (1\le j<k\le n)$. 
\item ${\rm deg}\ Z_{n+1,i}^{i}=1$ and   
${\rm deg}\ Z_{n+1,i}^{n+1}=n\ (i=1,2,\ldots,n)$, 

\noindent
${\rm deg}\ W_{jk}^{n+1}=2\ (1\le j<k\le n)$. 
\end{enumerate}
\end{Lemma}
\begin{figure}[htbp]
      \begin{center}
\scalebox{0.45}{\includegraphics*{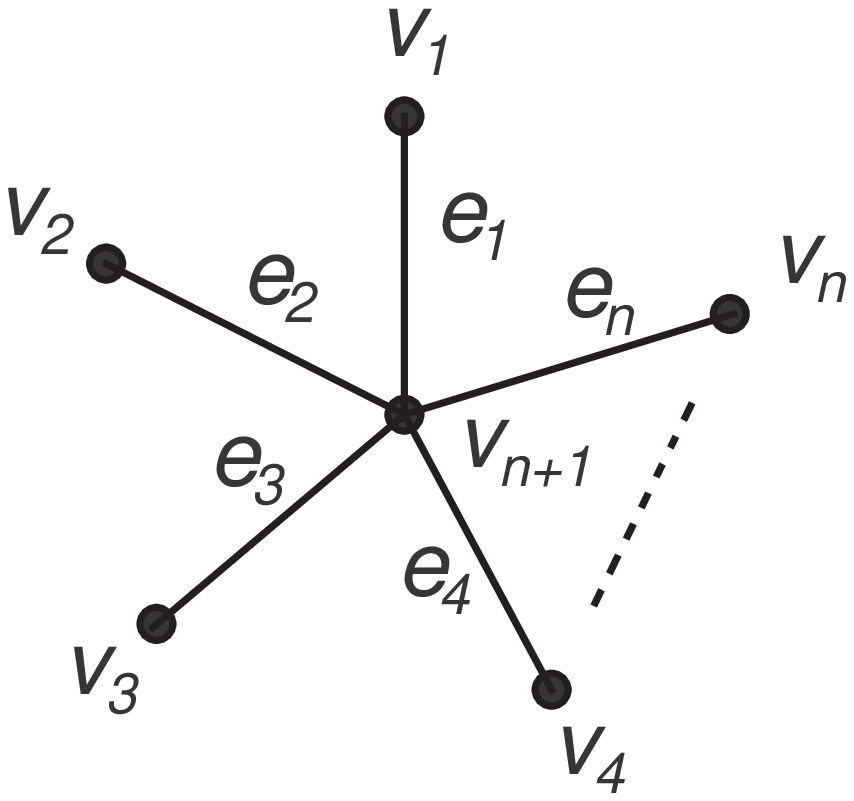}}
      \end{center}
   \caption{}
  \label{startuben}
\end{figure}
\begin{Example}\label{stube_ex}  
{\rm Fig. \ref{c_tube_ex} illustrates the symmetric tube $S_{n}^{(0)}$ of $S_{n}$ for $n=1,2$ and $3$. We can see that the subgraph $H$ of $S_{n}^{(0)}$ defined by 
\begin{eqnarray*}
V(H)&=&\{Z_{n+1,i}^{n+1}~(i=1,2,\ldots,n), 
\ W_{jk}^{n+1}\ (1\le j<k\le n)\},\\
E(H)&=&\{Y_{jk}^{n+1}, Y_{kj}^{n+1}\ (1\le j<k\le n)\}
\end{eqnarray*} 
is homeomorphic to $K_{n}$. Precisely speaking, $H$ is isomorphic to the graph which is obtained from $K_{n}$ by subdividing each edge of $K_{n}$ once. 

\begin{figure}[htbp]
      \begin{center}
\scalebox{0.375}{\includegraphics*{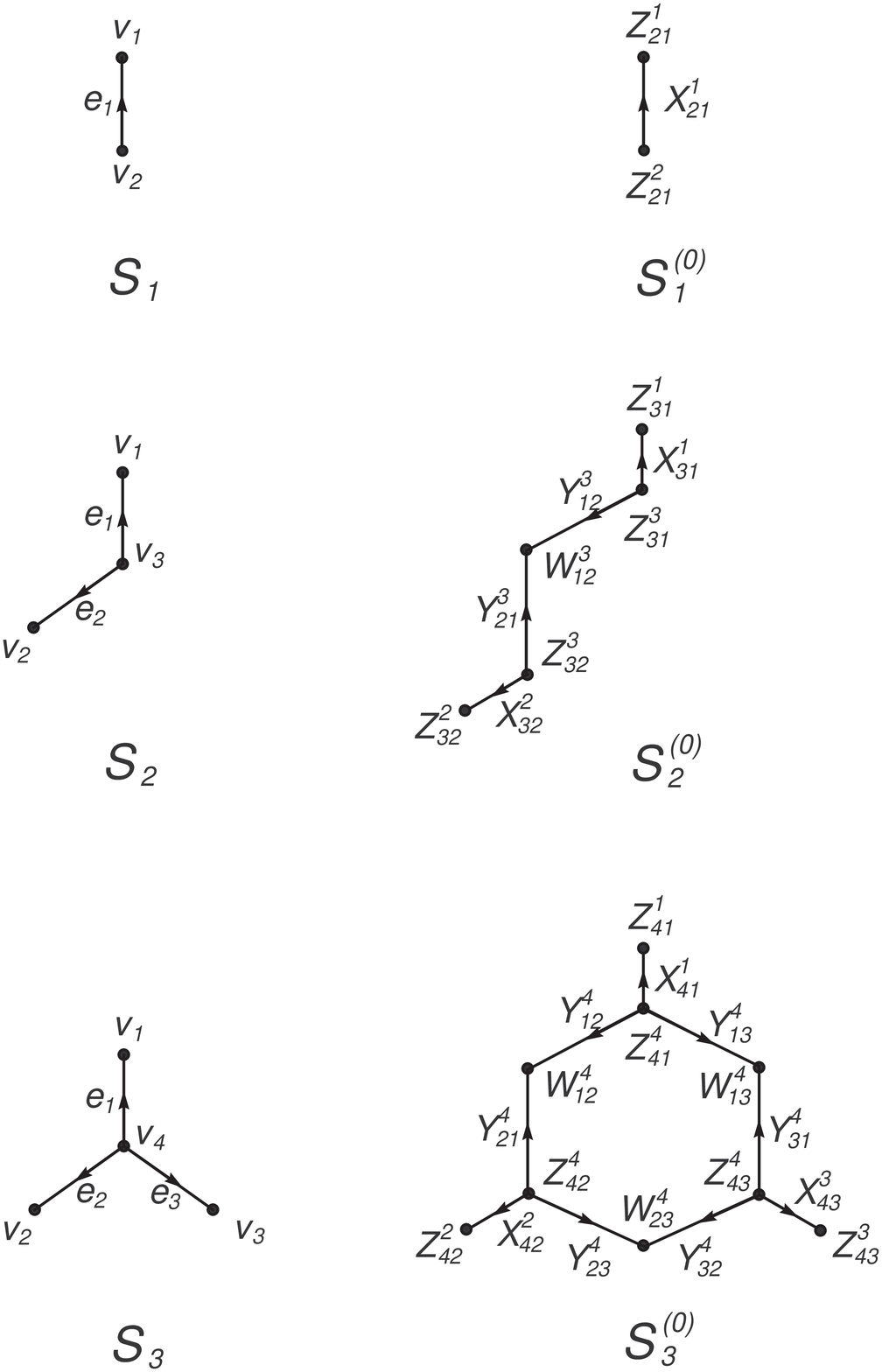}}
      \end{center}
   \caption{}
  \label{c_tube_ex}
\end{figure}
}
\end{Example}

By Lemma \ref{stube} and Example \ref{stube_ex}, we can see that the symmetric tube $G^{(0)}$ of a graph $G$ is obtained from $G$ by ``substituting'' $K_{d}$ for each vertex $v_{s}$ of $G$ as follows, where $d={\rm deg}v_{s}$. Also see Examples \ref{circle} and \ref{k4}.

\begin{Lemma}\label{main_lemma}  
Let $G$ be a graph. For a vertex $v_{s}$ of $G$, let $v_{s_{1}},v_{s_{2}},\ldots,v_{s_{d}}$ be all vertices connected to $v_{s}$ so that $e_{i_{l}}=(v_{s},v_{s_{l}})$ $(l=1,2,\ldots, d,\ 1\le i_{j}<i_{k}\le d\ {\rm for}\ j<k)$. Then $G^{(0)}$ is obtained from $G$ by replacing each $v_{s}$ with the graph $H_{s}$, which is homeomorphic to $K_{d}$ defined by 
\begin{eqnarray*}
V(H_{s})&=&\{Z_{ss_{l}}^{s}~(i=1,2,\ldots,d), 
\ W_{s_{j}s_{k}}^{s}\ (1\le j<k\le d)\},\\
E(H_{s})&=&\{Y_{s_{j}i_{k}}^{s}=(Z_{s s_{j}}^{s},W_{s_{j} s_{k}}^{s}), Y_{s_{k}i_{j}}^{s}=(Z_{s s_{k}}^{s},W_{s_{j} s_{k}}^{s})\ (1\le j<k\le d)\},
\end{eqnarray*} 
and by replacing $e_{i_{l}}$ with $X_{s s_{l}}^{i_{l}}=(Z_{s s_{l}}^{s_{l}},Z_{s s_{l}}^{s})\ (l=1,2,\ldots,d)$. 
\end{Lemma} 

By the universal coefficient theorem, it is sufficient to know $H_{1}(G^{(0)};{\mathbb Z})$ to calculate $H^{1}(G^{(0)};{\mathbb Z})$. So in the following we construct a spanning tree of $G^{(0)}$. First we define  a subgraph $T_{S_{n}^{(0)}}$ of $S_{n}^{(0)}$ as follows. We set $V(T_{S_{n}^{(0)}})=V(S_{n}^{(0)})$ and 
\begin{eqnarray*}
E(T_{S_{n}^{(0)}})&=&\{X_{n+1,i}^{i}\ (i=1,2,\ldots,n),\ Y_{nj}^{n+1}\ (j=1,2,\ldots,n-1), \\
&&Y_{tn}^{n+1}\ (t=1,2,\ldots,n-1),\ Y_{kl}^{n+1}\ (1\le k<l\le n-1)\}. 
\end{eqnarray*}
We note that 
\begin{eqnarray*}
E(S_{n}^{(0)}) \setminus E(T_{S_{n}^{(0)}})=\{Y_{lk}^{n+1}~(1\le k<l\le n-1))\}.
\end{eqnarray*} 
Then we easily have the following. 

\begin{Lemma}\label{sptree1}  
A subgraph $T_{S_{n}^{(0)}}$ is a spanning tree of $S_{n}^{(0)}$. 
\end{Lemma}

Now we construct a spanning tree of $G^{(0)}$ on the outcome of Lemma \ref{sptree1}. For $v_{s}\in V(G)$, let ${\rm st}(v_{s})$ be a subgraph of $G$ consisting of $v_{s}$ and all edges incident to $v_{s}$. Let $v_{s_{1}},v_{s_{2}},\ldots,v_{s_{d}}$ be all vertices connected to $v_{s}$ so that $e_{i_{l}}=(v_{s},v_{s_{l}})$ $(l=1,2,\ldots, d,\ 1\le i_{j}<i_{k}\le d\ {\rm for}\ j<k)$, where $d={\rm deg}v_{s}$. We construct a spanning tree $T_{{\rm st}(v_{s})}$ of ${\rm st}(v_{s})$ in the same way as $S_{n}^{(0)}$. Namely, $V(T_{{\rm st}(v_{s})})=V({\rm st}(v_{s}))$ and 
\begin{eqnarray*}
E(T_{{\rm st}(v_{s})})&=&\{X_{s s_{l}}^{i_{l}}\ (l=1,2,\ldots,d),\ Y_{s_{d}i_{j}}^{s}\ (j=1,2,\ldots,d-1),\\
&&Y_{s_{j}i_{d}}^{s}\ (j=1,2,\ldots,d-1),\ Y_{s_{j}i_{k}}^{s}\ (1\le j<k\le d-1)\}.
\end{eqnarray*}
Let $T_{G}$ be a spanning tree of $G$. We define a subgraph $T_{G^{(0)}}$ of $G^{(0)}$ by $V(T_{G^{(0)}})=V(G^{(0)})$ and 
\begin{eqnarray*}
E(T_{G^{(0)}})&=&
\{X_{j_{1}j_{2}}^{j}\ |\ e_{j}=(v_{j_{1}},v_{j_{2}})\in
E(T_{G})\}\\
&&\cup
\bigcup_{v_{s}\in V(G)}\left(E(T_{{\rm st}(v_{s})})\setminus \{X_{s s_{l}}^{i_{l}}\ (l=1,2,\ldots,d)\}\right).
\end{eqnarray*}
Then we have the following. 

\begin{Lemma}\label{sptree2}  
A subgraph $T_{G^{(0)}}$ is a spanning tree of $G^{(0)}$. 
\end{Lemma}

We note that 
\begin{eqnarray*}
E({\rm st}(v_{s})) \setminus E(T_{{\rm st}(v_{s})})
=\{Y_{s_{j}i_{k}}^{s}\ (1\le j<k\le d-1)\}
\end{eqnarray*}
for $v_{s}\in V(G)$ 
and this set is empty for $d=1,2$. Therefore we have that 
\begin{eqnarray*}
E(G^{(0)}) \setminus E(T_{G^{(0)}})&=&
\{X_{j_{1}j_{2}}^{j}\ |\ e_{j}=(v_{j_{1}},v_{j_{2}})\in E(G) \setminus E(T_{G})\}\\
&&\cup
\bigcup_{\mathop{{v_{s}\in V(G)}\atop{{\rm deg}v_{s}\ge 3}}}
\{Y_{s_{j}i_{k}}^{s}\ (1\le j<k\le d-1)\}.
\end{eqnarray*}
Thus we can determine a structure of $H^{1}(G^{(0)};{\mathbb Z})$ completely as follows. 

\begin{Theorem}\label{present} 
\begin{enumerate}
\item $\displaystyle H^{1}(G^{(0)};{\mathbb Z})\ {\cong}\bigoplus_{\mathop{{e_{j}\in E(G) \setminus E(T_{G})}\atop{e_{j}=(v_{j_{1}},v_{j_{2}})}}}\left<X_{j}^{j_{1}j_{2}}\right>\oplus \bigoplus_{\mathop{{v_{s}\in V(G)}\atop{{\rm deg}v_{s}\ge 3}}} \left(\bigoplus_{1\le j<k\le d-1} \left<Y_{s}^{s_{j}i_{k}}\right>\right)$,
\item $\displaystyle {\rm rank}\ H^{1}(G^{(0)};{\mathbb Z})=1-2n+\frac{1}{2}\sum_{s=1}^{m}({\rm deg}v_{s})^{2}$.   
\end{enumerate}
\end{Theorem}

\begin{proof}
(1) is clear. We show (2). We have that 
\begin{eqnarray*}
{\rm rank}~H^{1}(G^{(0)};{\mathbb Z})
&=&{\rm rank}\ H^{1}(G;{\mathbb Z})
+\sum_{s=1}^{m}\frac{1}{2}({\rm deg}v_{s}-1)({\rm deg}v_{s}-2)\\
&=&n-m+1
+\frac{1}{2}\sum_{s=1}^{m}\left\{({\rm deg}v_{s})^{2}
-3{\rm deg}v_{s}+2\right\}\\
&=&n-m+1
+\frac{1}{2}\sum_{s=1}^{m}({\rm deg}v_{s})^{2}
-\frac{3}{2}\sum_{s=1}^{m}{\rm deg}v_{s}
+m\\
&=&n+1
+\frac{1}{2}\sum_{s=1}^{m}({\rm deg}v_{s})^{2}
-3n\\
&=&1-2n+\frac{1}{2}\sum_{s=1}^{m}({\rm deg}v_{s})^{2}. 
\end{eqnarray*}
This completes the proof. 
\end{proof}

\begin{Example}\label{circle}
{\rm Let $K_{3}$ be the complete graph on three vertices as illustrated in the left side of Fig. \ref{circletube}. As we saw in Example \ref{cell_dec_ex}, the symmetric tube $K_{3}^{(0)}$ is a graph as illustrated in Fig. \ref{circletube}. For a spanning tree $T_{K_{3}}=e_{1}\cup e_{2}$ of $K_{3}$, by Theorem \ref{present} we have that $H^{1}({K}_{3}^{(0)};{\mathbb Z})=\left<X_{1}^{12}\right>\cong {\mathbb Z}$. We note that if $[x,y,1/2]\in {K}_{3}^{(0)}$ rotates once around the one in the direction induced by the orientation of $X_{12}^{1}$ then the non-ordered pair $(x,y)$ rotates once around $K_{3}$; see Fig. \ref{tanvec1}. Here the initial and terminal points of a vector in Fig. \ref{tanvec1} correspond to $x$ and $y$, respectively. This shows that ${\mathcal R}(f)$ of a plane curve $f$ coincides with the rotation number of $f$. 
}
\end{Example}
\begin{figure}[htbp]
      \begin{center}
\scalebox{0.35}{\includegraphics*{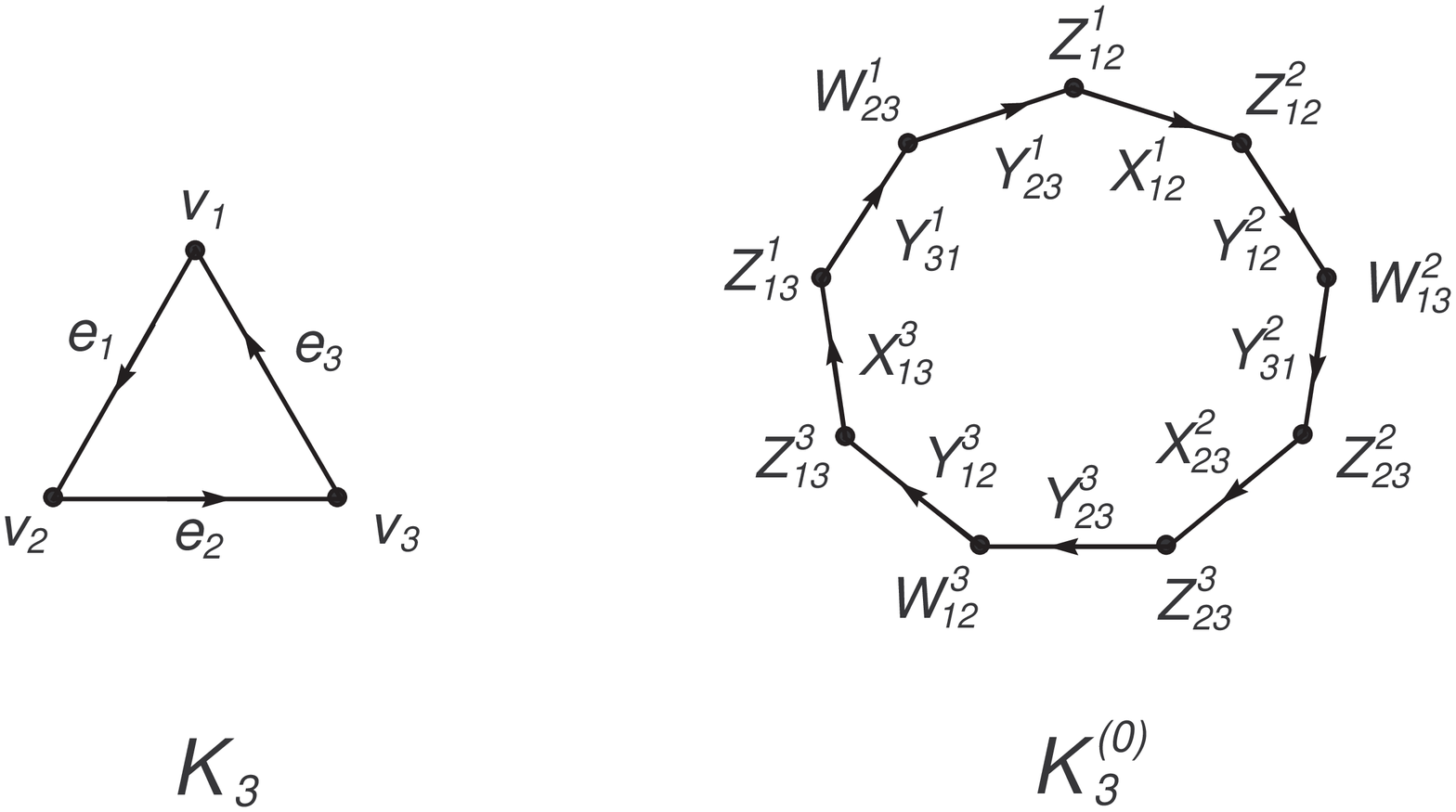}}
      \end{center}
   \caption{}
  \label{circletube}
\end{figure}
\begin{figure}[htbp]
      \begin{center}
\scalebox{0.325}{\includegraphics*{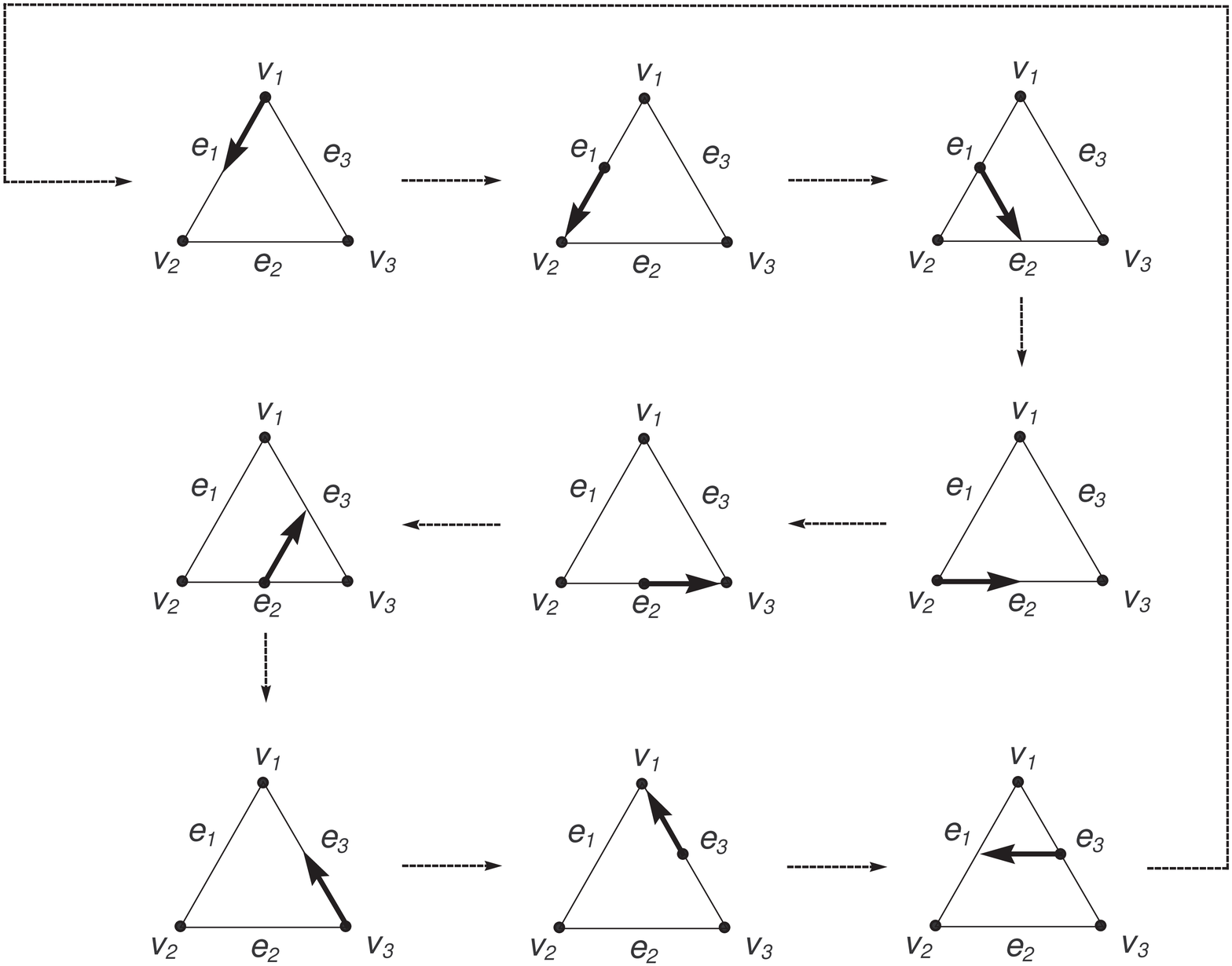}}
      \end{center}
   \caption{}
  \label{tanvec1}
\end{figure}
\begin{Example}\label{star}
{\rm For $S_{3}$ and its symmetric tube $S_{3}^{(0)}$ as illustrated in Fig. \ref{c_tube_ex}, by Theorem \ref{present} we have that $H^{1}({S}_{3}^{(0)};{\mathbb Z})=\left<Y_{4}^{21}\right>\cong {\mathbb Z}$. We note that if $[x,y,1/2]\in S_{3}^{(0)}$ rotates once around the cycle represented by $Y_{21}^{4}$ in the direction induced by the orientation of the one then the non-ordered pair $(x,y)$ rotates once around $v_{4}$; see Fig. \ref{tanvec2}. Here the initial and terminal points of a vector in Fig. \ref{tanvec2} correspond to $x$ and $y$, respectively.

Let $f$ be a generic plane immersion of $S_{3}$. Then there exists a neighbourhood $U$ of $v_{4}$ such that $f{|}_{U}$ is an embedding. We can see that ${\mathcal R}(f)=1$ if $f{|}_{U}(e_{1}\cap U)$, $f{|}_{U}(e_{2}\cap U)$ and $f{|}_{U}(e_{3}\cap U)$ are embedded in ${\mathbb R}^{2}$ as illustrated in Fig. \ref{vrotate} (1), and ${\mathcal R}(f)=-1$ if $f{|}_{U}(e_{1}\cap U)$, $f{|}_{U}(e_{2}\cap U)$ and $f{|}_{U}(e_{3}\cap U)$ are embedded in ${\mathbb R}^{2}$ as illustrated in Fig. \ref{vrotate} (2). 
}
\end{Example}
\begin{figure}[htbp]
      \begin{center}
\scalebox{0.325}{\includegraphics*{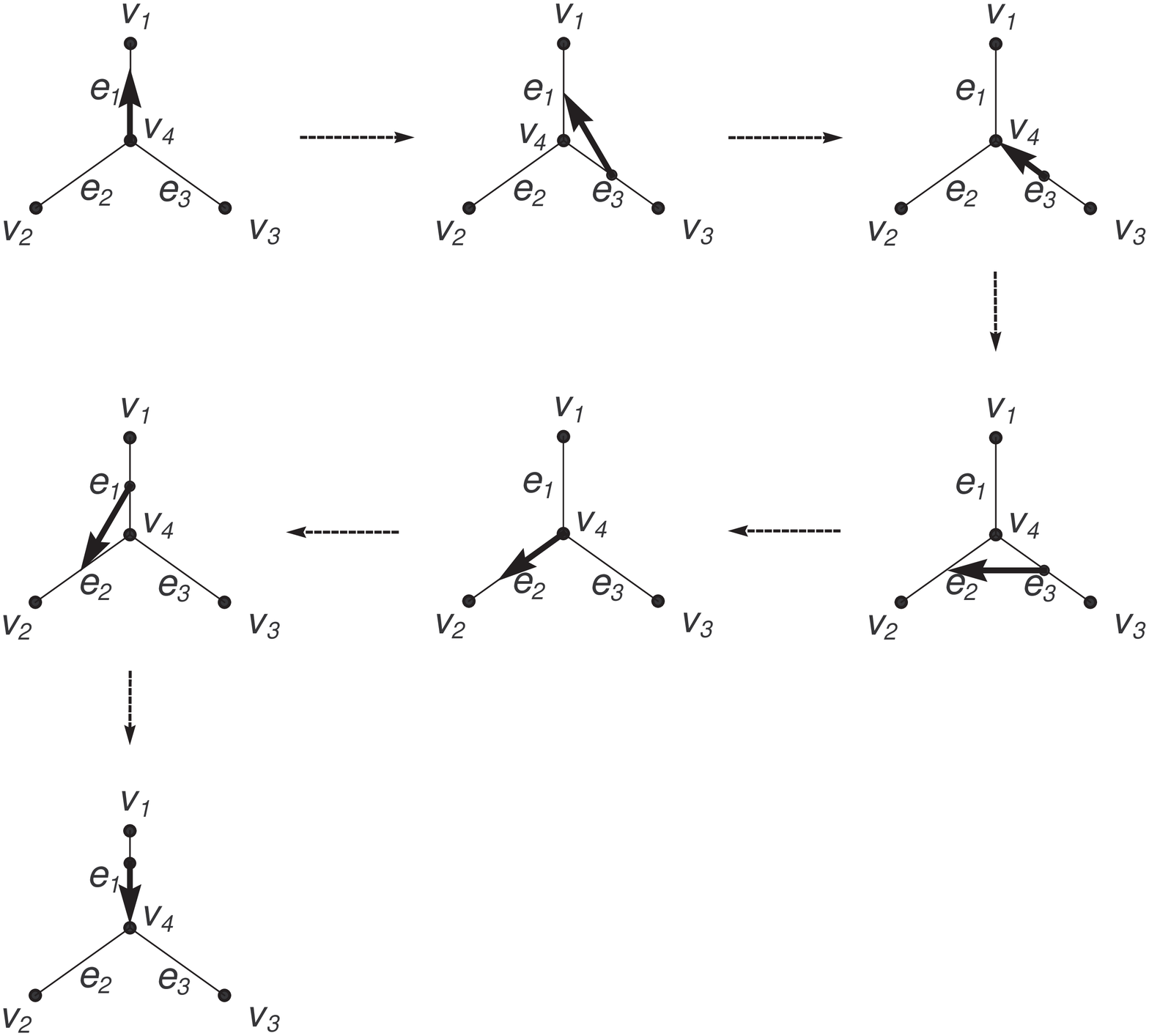}}
      \end{center}
   \caption{}
  \label{tanvec2}
\end{figure}
\begin{figure}[htbp]
      \begin{center}
\scalebox{0.35}{\includegraphics*{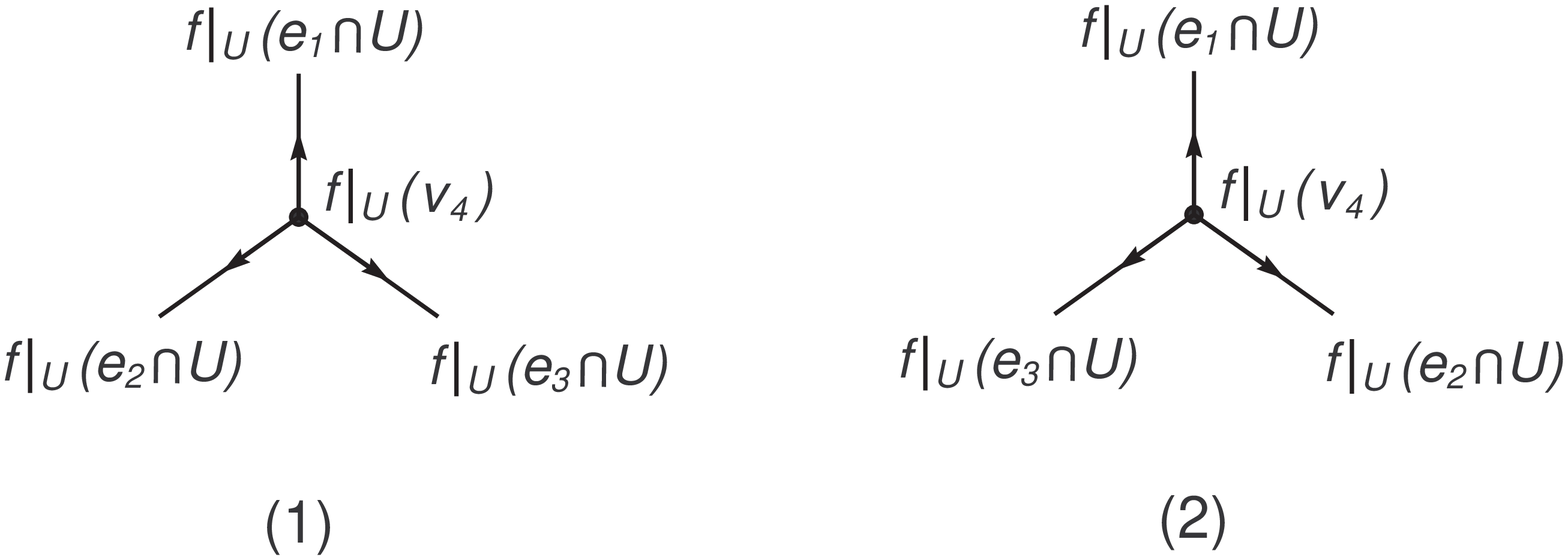}}
      \end{center}
   \caption{}
  \label{vrotate}
\end{figure}
\begin{Example}\label{k4}
{\rm  
Let $K_{4}$ be the complete graph on four vertices and $f$, $g$ and $h$ three generic plane immersions of $K_{4}$ as ilustrated in Fig. \ref{K4imm}. Then the symmetric tube of $K_{4}$ is a graph as illustrated in Fig. \ref{K4tube}. For a spanning tree $T_{K_{4}}=e_{1}\cup e_{2}\cup e_{3}$ of $K_{4}$, by Theorem \ref{present} we have that 
\begin{eqnarray*}
H^{1}({K}_{4}^{(0)};{\mathbb Z})
=\left<X_{4}^{23},X_{5}^{24},X_{6}^{34},Y_{1}^{31},
Y_{2}^{31},Y_{3}^{22},Y_{4}^{23}\right>\cong 
\underbrace{{\mathbb Z}\oplus{\mathbb Z}\oplus\cdots\oplus{\mathbb Z}}_{7~{\rm times}}.
\end{eqnarray*}
By calculating on the outcome of Examples \ref{circle} and \ref{star}, we have that 
\begin{eqnarray*}
{\mathcal R}(f)&=&(-1,1,-1,1,1,1,1),\\
{\mathcal R}(g)&=&(1,-1,1,1,-1,1,-1)~{\rm and} \\
{\mathcal R}(h)&=&(0,0,0,1,1,-1,1). 
\end{eqnarray*}
\begin{figure}[htbp]
      \begin{center}
\scalebox{0.35}{\includegraphics*{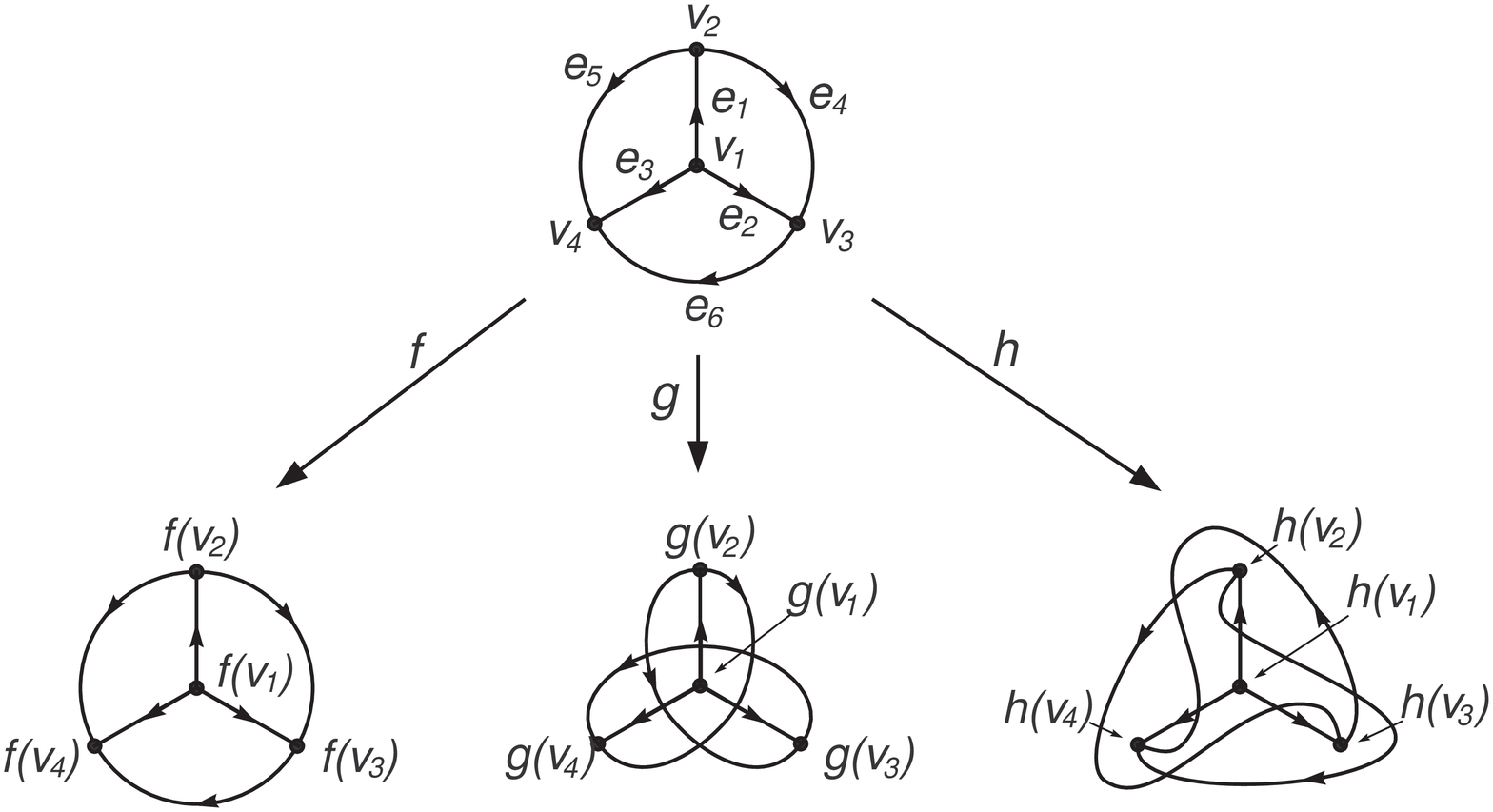}}
      \end{center}
   \caption{}
  \label{K4imm}
\end{figure}
\begin{figure}[htbp]
      \begin{center}
\scalebox{0.375}{\includegraphics*{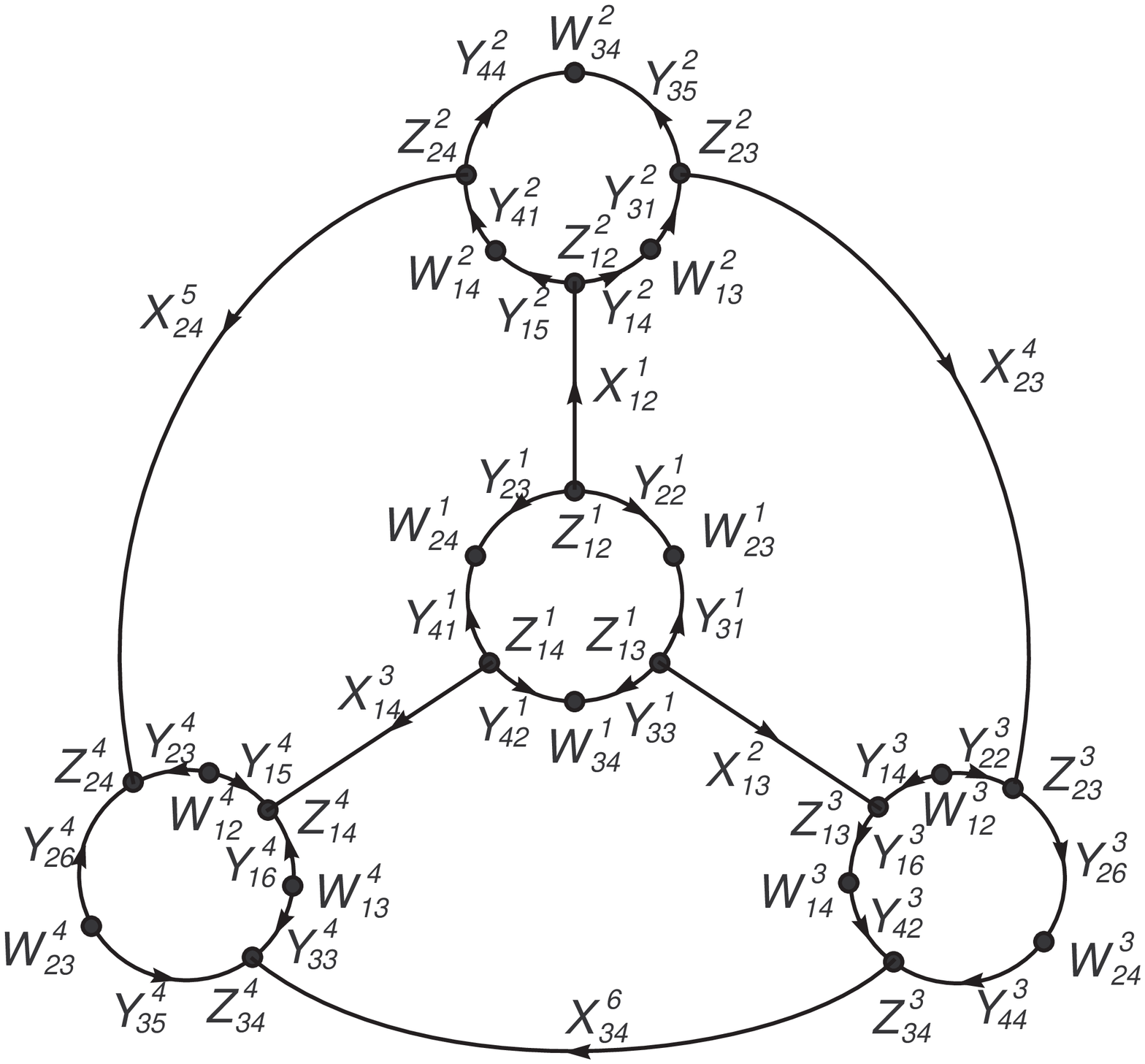}}
      \end{center}
   \caption{$K_{4}^{(0)}$}
  \label{K4tube}
\end{figure}
}
\end{Example}

\section{Proof of Theorem \ref{main}} 

First we show two lemmas which are needed to prove Theorem \ref{main}. 

\begin{Lemma}\label{wtrick} 
Each of the local moves as illustrated in Fig. \ref{whitneytrick} (4) and (5) are represented by a sequence of moves from the list as illustrated in Fig. \ref{rhomo2} (1), (2), (3) and ambient isotopies.  
\end{Lemma}

\begin{proof}
See Fig. \ref{whitneytrickpr} and \ref{whitneytrickpr2}. 
\end{proof}

\begin{figure}[htbp]
      \begin{center}
\scalebox{0.4}{\includegraphics*{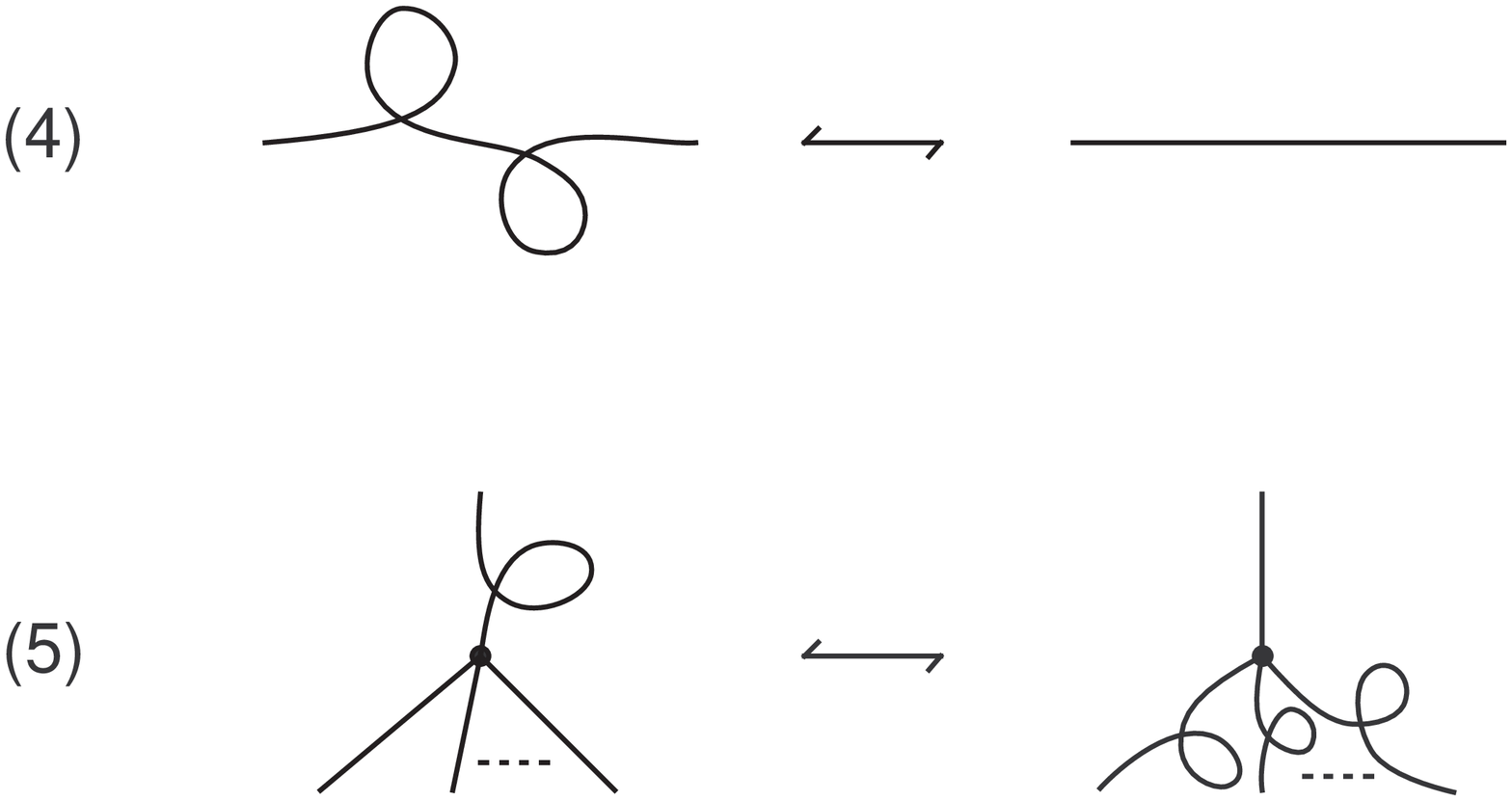}}
      \end{center}
   \caption{}
  \label{whitneytrick}
\end{figure}
\begin{figure}[htbp]
      \begin{center}
\scalebox{0.4}{\includegraphics*{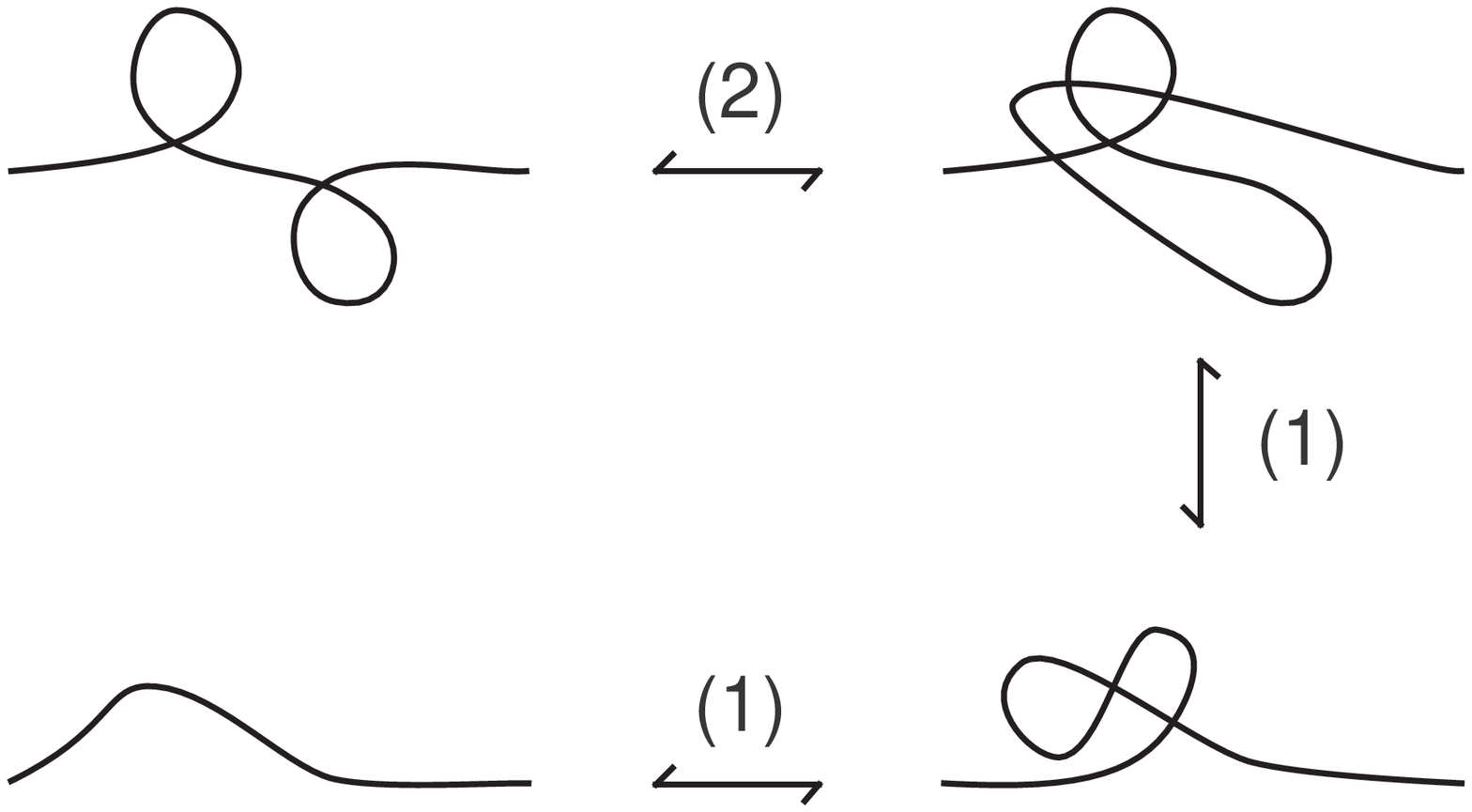}}
      \end{center}
   \caption{}
  \label{whitneytrickpr}
\end{figure}
\begin{figure}[htbp]
      \begin{center}
\scalebox{0.4}{\includegraphics*{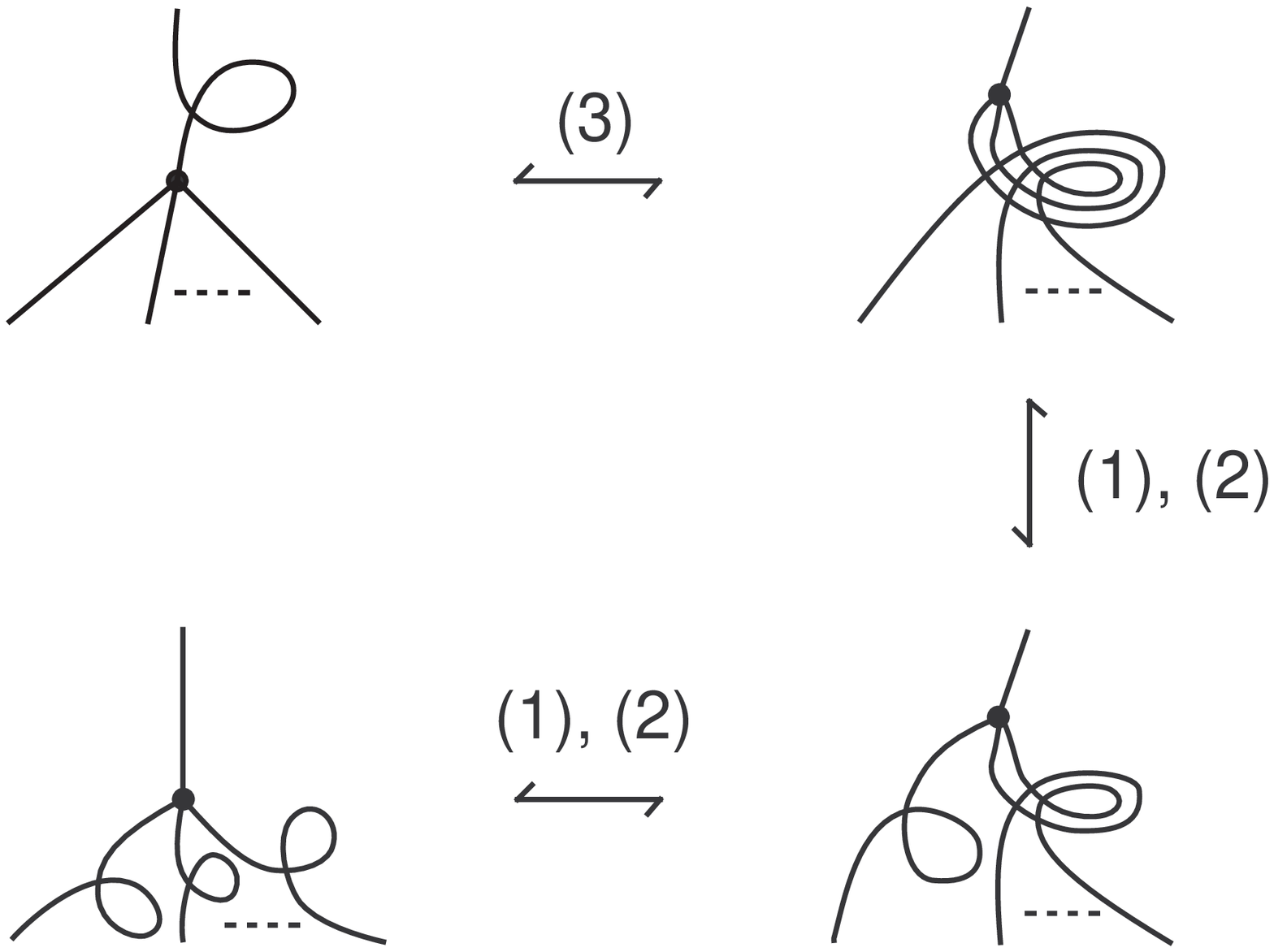}}
      \end{center}
   \caption{}
  \label{whitneytrickpr2}
\end{figure}

We remark here that the local move as illustrated in Fig. \ref{whitneytrick} (4) is none other than the Whitney trick. 

\begin{Lemma}\label{keylemma1} 
Let $G$ be a graph, $H$ a connected subgraph of $G$ and $f$ and $g$ two plane immersions of $G$. If ${\mathcal R}(f)={\mathcal R}(g)$, then ${\mathcal R}(f{|}_{H})={\mathcal R}(g{|}_{H})$. 
\end{Lemma}

\begin{proof}
Let $i:H\to G$ be the inclusion. Since $i$ is injective, the homomorphism 
\begin{eqnarray*}
i^{*}:H^{1}(G^{(0)};{\mathbb Z})\longrightarrow H^{1}(H^{(0)};{\mathbb Z})
\end{eqnarray*}
is induced by $i$. Clearly $i^{*}({\mathcal R}(f))={\mathcal R}(f{|}_{H})$ and $i^{*}({\mathcal R}(g))={\mathcal R}(g{|}_{H})$. Therefore by the assumption we have the desired conclusion. 
\end{proof}

For a generic plane immersion $f$ and a vertex $v_{s}$ of a graph $G$, a cyclic order of the edges of $G$ incident to $v_{s}$ is determined by considering a neighbourhood $U$ of $v_{s}$ so that $f|_{U}$ is an embedding. We call it a {\it cyclic order of $f(v_{s})$}.

\begin{proof}[Proof of Theorem \ref{main}.] 
Since (1)$\Rightarrow$(3) is shown by Proposition \ref{inva} and (2)$\Rightarrow$(1) is clear, it is sufficient to show that (3)$\Rightarrow$(2). Assume that ${\mathcal R}(f)={\mathcal R}(g)$. In the following we show that $f$ and $g$ are transformed into each other by the moves as illustrated in Fig. \ref{rhomo2} (1), (2), (3), Fig. \ref{whitneytrick} (4), (5) and ambient isotopies. Then by Lemma \ref{wtrick}, we have the desired conclusion. Since ${\mathcal R}(f)={\mathcal R}(g)$, by Lemma \ref{keylemma1} we have that ${\mathcal R}(f|_{{\rm st}(v_{s})})={\mathcal R}(g|_{{\rm st}(v_{s})})$ for any vertex $v_{s}$ of $G$. Then by Example \ref{star} we have that the cyclic order of $f(v_{s})$ is equal to the cyclic order of $g(v_{s})$ for any vertex $v_{s}$ of $G$. Let $T_{G}$ be a spanning tree of $G$. By using the moves as illustrated in Fig. \ref{rhomo2} (1), (2), (3), Fig. \ref{whitneytrick} (5) and ambient isotopies in case of necessity, we can deform $f$ (resp. $g$) so that $f|_{T_{G}}$ (resp. $g|_{T_{G}}$) is an embedding. Since the cyclic order of $f(v_{s})$ is equal to the cyclic order of $g(v_{s})$ for any vertex $v_{s}$ of $G$, we may assume that $f|_{T_{G}}=g|_{T_{G}}$. 
We set $E(G) \setminus E(T_{G})=\{e_{k_{1}},e_{k_{2}},\ldots,e_{k_{\beta}}\}$, 
where $\beta$ denotes the first Betti number of $G$. Let $p_{k_{i}}$ be the unique path on $T_{G}$ which connects the terminal vertices of $e_{k_{i}}$. We denote a cycle $e_{k_{i}} \cup p_{k_{i}}$ by ${\gamma}_{k_{i}}$. Note that the double points of $f|_{{\gamma}_{k_{i}}}$ (resp. $g|_{{\gamma}_{k_{i}}}$) are only the double points of $f|_{e_{k_{i}}}$ (resp. $g|_{e_{k_{i}}}$). Then, by using the moves as illustrated in Fig. \ref{rhomo2} (1), (2), (3), Fig. \ref{whitneytrick} (4) and ambient isotopies in case of necessity, we can deform $f|_{{\gamma}_{k_{i}}}$ into the generic plane immersion of $\gamma_{k_{i}}$ as illustrated in Fig. \ref{rhclass3} (1) or (2) ($i=1,2,\ldots,\beta$). Then, by Lemma \ref{keylemma1} we have that ${\mathcal R}(f{|}_{{\gamma}_{k_{i}}})={\mathcal R}(g{|}_{{\gamma}_{k_{i}}})$, namely $f{|}_{{\gamma}_{k_{i}}}$ and $g{|}_{{\gamma}_{k_{i}}}$ have the same rotation number. Thus we may assume that $f|_{{\gamma}_{k_{i}}}=g|_{{\gamma}_{k_{i}}}$ for $i=1,2,\ldots,\beta$. This implies that we can deform $f$ and $g$ identically by the moves as illustrated in Fig. \ref{rhomo2} (1), (2), (3), Fig. \ref{whitneytrick} (4), (5) and ambient isotopies. This completes the proof. 
\end{proof}
\begin{figure}[htbp]
      \begin{center}
\scalebox{0.4}{\includegraphics*{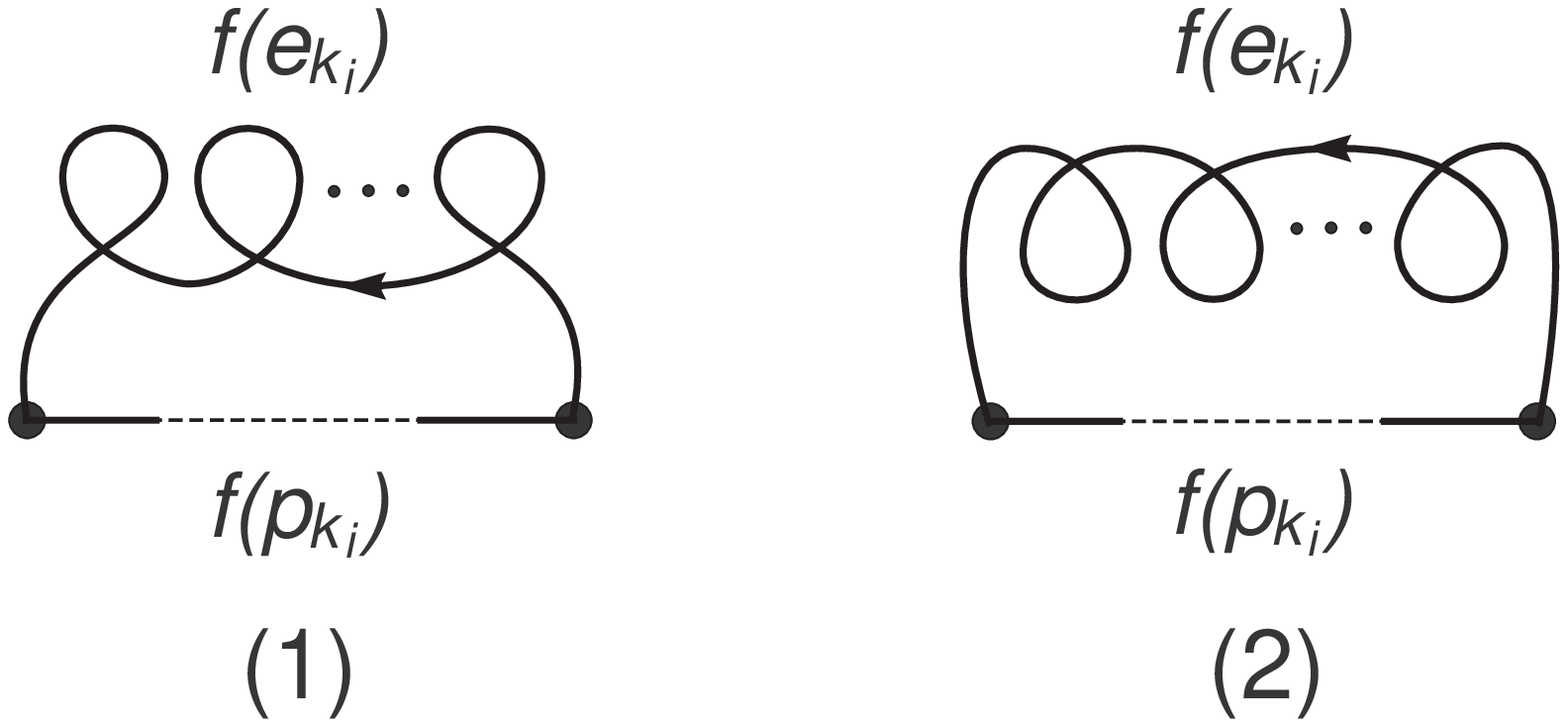}}
      \end{center}
   \caption{}
  \label{rhclass3}
\end{figure}

\section*{Acknowledgment}

The author is  most grateful to Professors Tohl Asoh and Tsutomu Yasui for their invaluable comments. 

\normalsize

\end{document}